\documentclass[10pt,dvips,twoside,letterpaper]{article}
\usepackage[usenames]{color}
\usepackage[english]{babel}

\usepackage{pslatex}
\usepackage{fancyhdr}
\usepackage{graphicx}
\usepackage{geometry}
\RequirePackage{amsfonts,amssymb,amsmath,amscd,amsthm}
\RequirePackage{txfonts} \RequirePackage{graphicx}
\RequirePackage{xcolor}

\def\figurename{Figure} % Replace the colon that normally appears after the Figure number by a period.
\makeatletter
\renewcommand{\fnum@figure}[1]{\figurename~\thefigure.}
\makeatother

\def\tablename{Table} % Replace the colon that normally appears after the Figure number by a period.
\makeatletter
\renewcommand{\fnum@table}[1]{\tablename~\thetable.}
\makeatother
\usepackage{amsmath}
\usepackage{amssymb}
\usepackage{amsfonts}
\usepackage{amsthm,amscd}

\theoremstyle{definition}

\theoremstyle{remark}

\numberwithin{equation}{section}

\newtheorem{dfn}{Definition}

\newtheorem{corol}{Corollary}

\newtheorem{theo}{Theorem}
\newtheorem{lem}{Lemma}
\begin{document}
\title{\bfseries\scshape{Blowing-up solutions of a time-space fractional semi-linear equation with a structural damping and a  nonlocal in time nonlinearity }}
\author{\bfseries\scshape K. Bouguetof\thanks{E-mail address: khaoula.bouguetof@univ-tebessa.dz}\\
Laboratory of Mathematics, Informatics and Systems\\University of Larbi Tebessi, Tebessa, Algeria.
}
\date{}
\maketitle
\begin{abstract}
In this paper, we investigate the semilinear equation with a
 time-space fractional structural damping and a  nonlocal in time nonlinearity
\begin{equation*}
{\mathbf{D}}_{0|t}^{1+\alpha _{1}}u+(-\Delta)^\sigma u+(-\Delta )^{\delta}\mathbf{D}%
_{0|t}^{\alpha _{2}}u=I_{0|t}^{1-\gamma }|u|^{p},\hspace*{0.5cm}(t,x)\in
(0,\infty )\times \mathbb{R}^{N},
\end{equation*}%
where\thinspace $p>1,\alpha _{i},\gamma\in (0,1),\delta, \sigma \in (0,1)$\thinspace,\thinspace ${\mathbf{D}}_{0|t}^{\alpha _{i}}$\thinspace
is the Caputo fractional derivative and\thinspace $I_{0|t}^{1-\gamma }$%
\thinspace is the Riemann-Liouville fractional integral of order\thinspace $%
1-\gamma .$ We prove the non-existence of global solutions  if\thinspace 
\begin{equation*}
1<p\leqslant \frac{2(2+\alpha_1-\gamma) }{(\frac{\alpha_1+1}{\sigma} N+2\gamma-2\alpha_1-2)_+ }+1,
\end{equation*}
for any space dimension $N\geqslant 1.$ Then, we extend the result to the
system 
\begin{align*}
&{\mathbf{D}}_{0|t}^{1+\alpha _{1}}u+(-\Delta)^{\sigma_1} u+(-\Delta )^{\delta _{1}}{\mathbf{D}}%
_{0|t}^{\alpha _{2}}u=I_{0|t}^{1-\gamma _{1}}|v|^{p},\hspace*{0.5cm}%
(t,x)\in (0,\infty )\times \mathbb{R}^{N}, \\ 
&{\mathbf{D}}_{0|t}^{1+\beta _{1}}v+(-\Delta)^{\sigma_2} v+(-\Delta )^{\delta _{2}}{\mathbf{D}}%
_{0|t}^{\beta _{2}}v=I_{0|t}^{1-\gamma _{2}}|u|^{q},\hspace*{0.5cm}%
(t,x)\in (0,\infty )\times \mathbb{R}^{N},  
\end{align*}%
where\thinspace $p,q>1,0<\delta _{i},\sigma_i<1$\thinspace and\thinspace $\gamma
_{2}\in (0,1).$ Also, we present the
necessary conditions for the existence of local or global solutions.
\end{abstract}
\textit{Keywords  and phrase:} Cauchy problem,  time-space fractional derivatives, structural damping, nonexistence
\section{Introduction}

\hspace*{0.3cm} We consider the following Cauchy problem 
\begin{equation}
{\mathbf{D}}_{0|t}^{1+\alpha _{1}}u+(-\Delta)^{\sigma} u+(-\Delta )^{\delta }({}\mathbf{%
D}_{0|t}^{\alpha _{2}}u)=I_{0|t}^{1-\gamma }|u|^{p},\hspace*{0.5cm}%
(t,x)\in (0,\infty )\times \mathbb{R}^{N},
\label{a}
\end{equation}%
with the initial conditions
\begin{equation}\label{aa}
u(0,x)=u_{0}, \;u_{t}(0,x)=u_{1}(x),\hspace*{0.5cm}x\in \mathbb{R}^{N},%
\end{equation}
where\thinspace $u=u(t,x),p>1,0<\gamma<\alpha _{2}<\alpha _{1}<1,0<\sigma <\delta <1$%
\thinspace\ and\thinspace ${\mathbf{D}}_{0|t}^{\alpha }u=I_{0|t}^{1-\alpha
}u_{t}$, $I_{0|t}^{1-\alpha }$\thinspace\ is the Riemann-Liouville
fractional integral of order \thinspace $1-\alpha $\thinspace\  which is defined for $u\in C(0,t)$
as following
\begin{equation*}
I_{0|t}^{1-\alpha }u=\frac{1}{\Gamma (\alpha )}\int_{0}^{t}(t-\tau )^{-\alpha
}u(\tau )d\tau .
\end{equation*}%
The fractional Laplacien operator is defined as
\begin{equation*}
(-\Delta )^{\delta }u(t,x)=\frac{C(N,\delta)}{2}\int_{\mathbb{R}^N}\frac{2u(t,x)-u(t,x+y)-u(t,x-y)}{\vert y\vert^{N+2\delta}} dy,
\end{equation*}%
where $C(N,\delta)$ is a positive normalizing constant depending on $N$ and $\delta.$
 The term \thinspace $(-\Delta )^{\sigma}\mathbf{D%
}_{0|t}^{\alpha _{2}}u$\thinspace\ represents a generalized structural damping.\newline
\hspace*{0.3cm} For the semilinear damped wave equation
\begin{equation}\label{j}
u_{tt}-\Delta u+u_{t}=|u(t)|^{p},\hspace*{0.5cm}(t,x)\in (0,\infty )\times 
\mathbb{R}^{N}.
\end{equation}
Todorova and Yordanov\thinspace\cite{todorova2001critical} studied the global existence of mild solutions to $(\ref{j}),$ be it valid, if $p>1+\frac{2}{N}$ with $\Vert u_0\Vert$ and $\Vert u_1\Vert$  sufficiently small. In addition, they proved that the mild solution cannot  exist globally when $1<p< 1+\frac{2}{N}$ and  $\int u_i>0, i=0,1$.
 Here, the critical exponent of the damped wave equation $(\ref{j})$ is equal to the Fujita
critical exponent for\thinspace $u_{t}-\Delta u=|u|^{p}$.\newline
\hspace*{0.3cm}M. Kirane and N. Tatar\thinspace\cite{kirane2006nonexistence} studied the particular case
of\thinspace $(\ref{a})$ with $\alpha _{1}=\gamma =1$\thinspace and\thinspace $%
\delta =0$, they discussed the nonexistence of global solutions and
established conditions for which the problem has no local weak solution.\newline
\hspace*{0.3cm}D'Abbicco and Ebert\thinspace\cite{d2017global} considered a more
general case, they treated the following problem 
\begin{equation*}
u_{tt}+(-\Delta )^{\sigma }u+(-\Delta )^{\delta }u_{t}=|u(t)|^{p},\hspace*{%
0.5cm}(t,x)\in (0,\infty )\times \mathbb{R}^{N},
\end{equation*}%
where\thinspace $2\delta \leqslant \sigma .$ They proved that if $%
1<p\leqslant 1+\frac{2\sigma }{(N-2\delta )}$\thinspace , then the solution
blows-up in a finite time.
\newline
\hspace*{0.3cm} Motivated by the above results, the paper presents the results
for the nonexistence of global solutions to problem\thinspace $(\ref{a})-(\ref{aa})$%
\thinspace . The analysis is based on the \textit{test function method}.%
\newline
\hspace*{0.3cm}The rest of the paper is organized as follows: In section 2, we recall some definitions about fractional calculus which will be used in the sequel. In Section \thinspace 3\thinspace, we study the
absence of global weak solutions. In Section\thinspace 4\thinspace, we extend the previous results of Section 3 to a
system of semilinear coupled equations. In Section\thinspace\ 5\thinspace , we
establish the necessary conditions for local or global existence of problem\ $(\ref%
{a})-(\ref{aa})$.

\section{Preliminary}
\hspace*{0.3cm}In this section, we present some
results and basic properties of fractional calculus.  \\
\hspace*{0.3cm} For $0< \alpha< 1$ and $T>0$, the Riemann-Liouville derivatives of
 order $\alpha$ for a continous function $f$ are defined as
\begin{align*}
D^{\alpha}_{0\vert t}f(t)=\frac{1}{\Gamma(1-\alpha)}
\frac{d}{dt}\int_0^t
(t-\tau)^{-\alpha}f(\tau)~d\tau,\hspace*{0.3cm} D^{\alpha}_{t\vert T}f(t)=-\frac{1}{\Gamma(1-\alpha)}
\frac{d}{dt}\int_t^T
(\tau-t)^{-\alpha}f(\tau)~d\tau.
\end{align*}
\hspace*{0.3 cm} For $0<\alpha< 1$ and $f\in AC[0,T]$, the
Caputo derivatives of  fractional order $\alpha$ are
defined as
\begin{align*}
\textbf{D}^{\alpha}_{0\vert t} f(t)=
\frac{1}{\Gamma(1-\alpha)}\int_0^t
(t-\tau)^{-\alpha}f^{'}(\tau)~d\tau,\hspace*{0.3cm}\textbf{D}^{\alpha}_{t\vert T} f(t)=-
\frac{1}{\Gamma(1-\alpha)}\int_t^T
(\tau-t)^{-\alpha}f^{'}(\tau)~d\tau.
\end{align*}
\hspace*{0.3cm} Assume $ \textbf{D}^{\alpha}_{0\vert t} f\in L^1(0,T)$, $g\in C^1([0,T])$ and $g(T)=0$, then we have the following formula of integration by parts
\begin{align*}%\label{h}
\int_0^T g(t)\textbf{D}^{\alpha}_{0\vert t} f(t) dt=\int_0^T (f(t)-f(0))\textbf{D}^{\alpha}_{t\vert T} g(t) dt.
\end{align*}
\begin{lem}
\begin{enumerate}
\item \thinspace Suppose\thinspace $f\in C(0,\infty )$\thinspace,
\thinspace $p\geqslant q>0$, and \thinspace $D_{t|T}^{p-q}f(t)$ \thinspace exists, then for \thinspace $t>0$
\begin{equation}\label{nb}
D_{t|T}^{p}(I_{t|T}^{q}f)(t)=D_{t|T}^{p-q}f(t).
\end{equation}%
In particular, when $p=n$ we have
\begin{equation*}
D_{t|T}^{n}(I_{t|T}^{q}f)(t)=(-1)^{n}D_{t|T}^{n-q}f(t).
\end{equation*}

\item \thinspace Let \,$n-1\leqslant q<n$, then for every $t>0$, 
\begin{align}\label{bo}
I^{m-p}_{t\vert T}D^q_{t\vert T} f(t)= D^{p+q-m}_{t\vert T}f(t)-\sum_{k=1}^n (-1)^{n-k}\frac{(T-t)^{m-p-k}D^{q-k}_{t\vert T}f(T)}{\Gamma(q-k)\Gamma(m-k-p+1)}.
\end{align}
\end{enumerate}
\end{lem}

\begin{proof}
\begin{enumerate}
\item Let $m-1\leqslant p<m$ and $n-1\leqslant p-q<n$, we have
\begin{align*}
D^p_{t\vert T}(I^q_{t\vert T}f)(t)
&=(-1)^m D^nD^{m-n}(I^{m-p+q}_{t\vert T}f)(t)\\
&=(-1)^{n}D^n(I^{n-p+q}_{t\vert T}f)(t)\\
&= D^{p-q}_{t\vert T}f(t).
\end{align*}
\item Using $(\ref{nb})$, we can write 
\begin{align*}
I^{m-p}_{t\vert T}D^q_{t\vert T} f(t)&=D^{q+p-m}_{t\vert T}\lbrace I^q_{t\vert T}D^q_{t\vert T} f(t)\rbrace\\
&=D^{q+p-m}_{t\vert T}\bigg\lbrace f(t)-\sum_{k=1}^n(-1)^{n-k}\frac{(T-t)^{q-k}D^{q-k}_{t\vert T}f(T)}{\Gamma(q-k+1)}\bigg\rbrace\\
&=D^{q+p-m}_{t\vert T} f(t)-\sum_{k=1}^n (-1)^{n-k}\frac{(T-t)^{m-p-k}D^{q-k}_{t\vert T}f(T)}{\Gamma(q-k)\Gamma(m-k-p+1)},
\end{align*}
due to the  following equality \thinspace\cite{yong2016basic}
\begin{align*}
I^q_{t\vert T} D^q_{t\vert T}f(t)= f(t)-\sum_{k=1}^n(-1)^{n-k}\frac{(T-t)^{q-k}D^{q-k}_{t\vert T}f(T)}{\Gamma(q-k+1)}.
\end{align*}
\end{enumerate}
\end{proof}

\begin{lem}
\label{k} Let $p$ and $q$ be real numbers, if $m-1\leqslant
p<m $ and $n-1\leqslant q<n$, then 
\begin{align}  \label{o}
D^p_{t\vert T} D^q_{t\vert T}f(t)=D^{p+q}_{t\vert T}f(t)-\sum_{k=1}^n(-1)^{n-k}\frac{%
(T-t)^{-p-k}D^{q-k}_{t\vert T}f(T)}{\Gamma(q-k)\Gamma(1-k-p)(m-k-p+1)}.
\end{align}
\end{lem}

\begin{proof}
 By the semigroup property of fractional integrals and (\ref{bo}), we can write 
\begin{align*}
D^p_{t\vert T} D^q_{t\vert T}f(t)&=D^m\left[(-1)^m I^{m-p}_{t\vert T}D^q_{t\vert T}f(t)\right]\\
&=D^{m}(-1)^m\left[D^{p+q-m}_{t\vert T}f(t)-\sum_{k=1}^n C_k{(T-t)^{m-p-k}D^{q-k}_{t\vert T}f(T)}\right]\\
&=D^{p+q}_{t\vert T}f(t)-\sum_{k=1}^n (-1)^{n-k}\frac{%
(T-t)^{-p-k}D^{q-k}_{t\vert T}f(T)}{\Gamma(q-k)\Gamma(1-k-p)(m-k-p+1)}.
\end{align*}
\end{proof}
For $T>0$ and $\eta\gg 1$, if we set 
\begin{equation}
\varphi _{1}(t)=%
\begin{cases}
\bigg(1-\frac{t}{T}\bigg)^{\eta }, & \hspace*{1cm}0<t\leqslant T, \\ 
0, & \hspace*{1cm}t> T,%
\end{cases}
\label{ob}
\end{equation}%
then 
\begin{align*}
D_{t|T}^{\alpha }\varphi _{1}(t)&=\frac{(\eta+1 )
}{\Gamma (\eta+1-\alpha )}T^{-\alpha }\bigg(1-\frac{t}{T}\bigg)^{\eta -\alpha }.
\end{align*}%
\begin{lem}[\cite{furati2008necessary}]\label{jii}
\label{ui} Let\thinspace $\varphi _{1}$\thinspace as in\thinspace $(\ref%
{ob})$, for \thinspace $\eta >\frac{p}{p-1}\theta -1$,
\begin{equation*}
\int_{0}^{T}D_{t|T}^{\theta }\varphi _{1}=\frac{C_{1}}{\eta -\theta }%
T^{1-\theta },
\end{equation*}%
and 
\begin{equation*}
\int_{0}^{T}\varphi _{1}^{-p^{\prime }/p}|D_{t|T}^{\theta }\varphi
_{1}|^{p^{\prime }}=CT^{1-p^{\prime }\theta },
\end{equation*}%
where 
\begin{equation*}
C=\frac{\eta ^{p^{\prime }}}{\eta +1-p^{\prime }\theta }\left[ \frac{\Gamma
(\eta -\theta ))}{\Gamma (\eta +1-2\theta )}\right] ^{p^{\prime }}.
\end{equation*}%
\end{lem}
\begin{lem}[\cite{bonforte2014quantitative}]
\label{kk} Let $\psi\in C^2(\mathbb{R}^N) $ be a real function decreasing in $\vert \xi\vert>1.$ Assume that
\begin{enumerate}
\item[i-] $\psi>0$ is compactly supported,
\item[ii-] $\psi\leqslant \vert \xi\vert ^{-\alpha}$ for $0<\alpha< N+2p\sigma $ and  $\vert \xi\vert$ large enough.
\end{enumerate}
Then there exist $C_1>0$ and $C_2>0$ such that
\begin{align}\label{ku}
\vert (-\Delta)^\sigma \psi\vert \leqslant C_1 \vert \xi\vert ^{-(N+2\sigma)} \hspace*{0.5cm}\text{and}\hspace*{0.5cm} \bigg \vert \frac{(-\Delta)^\sigma \psi}{\psi^{-1/p}}\bigg \vert^{\frac{p}{p-1}}< C_2.
\end{align}

\end{lem}
\section{Blow-up of solutions}
\hspace*{0.3cm} In this section, we first give the definition
of  weak solution of $(\ref{a})-(\ref{aa})$.  After  we prove the blow-up of solutions.
\begin{dfn}
\label{ab}Let $0<\alpha _{i}<1$, $p>1.$ For\thinspace\ $%
u_{0}, u_{1}\in L_{loc}^{p}(\mathbb{R}^{N})$, the function \thinspace $u\in
L_{loc}^{p}(Q_{T})$\thinspace\ is a weak solution of problem\thinspace $(\ref{a})-(\ref{aa})$ if 
\begin{align}
& \int_{Q_{T}}uD_{t|T}^{1+\alpha _{1}}\phi +\int_{Q_{T}}u(-\Delta)^\sigma \phi
+\int_{Q_{T}}u(-\Delta )^{\delta}D_{t|T}^{\alpha _{2}}\phi \notag  \label{ak} 
\nonumber\\&=\int_{Q_{T}}I_{0|t}^{1-\gamma }|u|^{p}\phi  
 +\int_{\mathbb{R^{N}}}u_{0}D_{t|T}^{\alpha _{1}}\phi
(0)+\int_{Q_{T}}u_{1}D_{t|T}^{\alpha _{1}}\phi +\int_{Q_{T}}u_{0}(-\Delta
)^{\delta}D_{t|T}^{\alpha _{2}}\phi,
\end{align}%
for  $Q_{T}:=[0,T]\times \mathbb{R^{N}}$, $\phi >0$, $\phi \in C_{c}^{\infty }([0,T]\times
\mathbb{R^{N}})$  with $\text{supp}_x\phi\subset\mathbb{R}^N$ and $\phi (T,.)=0.$
\end{dfn}

\begin{theo}
 Assume that $u_0=0$, $u_1\in L^1(\mathbb{R}^N)$ and $u_1\geqslant0$.
 If
\begin{equation*}
1<p\leqslant p*:= \frac{2(2+\alpha_1-\gamma) }{(\frac{\alpha_1+1}{\sigma} N+2\gamma-2\alpha_1-2)_+ }+1,
\end{equation*}%
then any solution to \thinspace(\ref{a})-(\ref{aa})\thinspace blows up in a finite
time.
\end{theo}
\begin{proof}
We assume the contrary. Let
\begin{align*}
\phi(t,x)=D^{1-\gamma}_{t\vert T}\tilde{\varphi}(t,x)=D^{1-\gamma}_{t\vert T}\left(\varphi_1(t)\varphi_2(x)\right),
\end{align*}
where $\varphi_{1}(t) $\thinspace and\thinspace$\varphi_2(x)=\psi(T^{-\theta/2}x)$\thinspace are defined as in  $(\ref{ob})$ and $(\ref{ku})$ .\\
According to\thinspace$(\ref{ak})$ and Lemma\thinspace \ref{k}, we have
\begin{align}\label{p}
&\int_{Q_T}\vert u\vert^p\tilde{\varphi}+\int_{Q_T}u_1\varphi_2 D^{1+\alpha_1-\gamma}_{t\vert T}\varphi_1\nonumber\\
&=\int_{Q_T} u\varphi_2D^{2+\alpha_1-\gamma}_{t\vert T}\varphi_1+\int_{Q_T} u(-\Delta)^{\delta}\varphi_2D^{1+\alpha_2-\gamma}_{t\vert T}\varphi_1
+\int_{Q_T} u(-\Delta)^\sigma\varphi_2D^{1-\gamma}_{t\vert T}\varphi_1.
\end{align}
%Under the condition\thinspace$\int_{\mathbb{R}^N}u_1(x)dx>0$, we obtain 
%\begin{align}\label{cd}
%\int_{Q_T}u_1\varphi_2 D^{1+\alpha_1-\gamma}_{t\vert T}\varphi_1
%&=C_1\int_{\mathbb{R}^N}u_1> 0,
%\end{align}
%where
%\begin{align*}
%C_1=\frac{(1-\alpha_1+\gamma)\Gamma(\eta+1)}{\Gamma(2-\alpha_1+\gamma+\eta)}.
%\end{align*}
Therefore, by Lemma\thinspace\ref{jii}, we get 
\begin{align}\label{h}
&\int_{Q_T}\vert u\vert^p\tilde{\varphi}+C T^{-(\alpha_1-\gamma)}\int_{\mathbb{R}^N}u_1(x)\varphi_2(x)dx\nonumber\\
&\leqslant \int_{Q_T} \vert u\vert \left(\varphi_2 D^{2+\alpha_1-\gamma}_{t\vert T}\varphi_1+D^{1+\alpha_2-\gamma}_{t\vert T}\varphi_1\vert (-\Delta)^{\delta}\varphi_2\vert+D^{1-\gamma}_{t\vert T}\varphi_1\vert (-\Delta)^\sigma\varphi_2\vert\right)\nonumber.
\end{align}
 Using the\thinspace Young inequality with parameters \thinspace$p$\thinspace and\thinspace$ p'=\frac{p}{p-1}$, we obtain
 \begin{align}
 &\int_{Q_T}\vert u\vert^p\tilde{\varphi}+C T^{-(\alpha_1-\gamma)}\int_{\mathbb{R}^N}u_1(x)\varphi_2(x)dx\nonumber\\
&\leqslant\frac{1}{3p}\int_{Q_T}\vert u\vert^p \tilde{\varphi}+\frac{3^{p'-1}}{p'}\int_{\Sigma}\varphi_2^{-p'/p}\varphi_1^{-p'/p}\vert D^{2+\alpha_1-\gamma}_{t\vert T}\varphi_1\vert^{p'}\nonumber\\&+\frac{1}{3p}\int_{Q_T}\vert u\vert^p \tilde{\varphi}+\frac{3^{p'-1}}{p'}\int_{\Sigma}\varphi_2^{-p'/p}\varphi_1^{-p'/p}\vert (-\Delta)^{\delta}\varphi_2 \vert^{p'}\vert D^{1+\alpha_2-\gamma}_{t\vert T}\varphi_1\vert^{p'}\nonumber\\&+\frac{1}{3p}\int_{Q_T}\vert u\vert^p \tilde{\varphi}+\frac{3^{p'-1}}{p'}\int_{\Sigma}\varphi_2^{-p'/p}\varphi_1^{-p'/p}\vert (-\Delta)^\sigma\varphi_2\vert^{p'}\vert D^{1-\gamma}_{t\vert T}\varphi_1\vert^{p'},
 \end{align}
where\thinspace$\Sigma=[0,T]\times supp\; \varphi_2.$ Using Lemma\thinspace\ref{kk}, it holds
\begin{align*}
\bigg \vert \frac{(-\Delta)^\sigma \varphi_2(x)}{\varphi_2^{1/p}}\bigg \vert^{p'}=T^{-\theta\sigma p'}\bigg \vert \frac{(-\Delta)^\sigma \psi}{\psi^{1/p}}\bigg \vert^{\frac{p}{p-1}}<T^{-\theta\sigma p'} C_2.
\end{align*}
Passing to the scaled variables
 \begin{align*}
 \tau=\frac{t}{T}\hspace*{0.2cm},\hspace*{0.2cm} \xi= \frac{\vert x\vert}{T^{\theta/2}} ,\hspace*{0.2cm} \theta=\frac{\alpha_1+1}{\sigma}\hspace*{0.2cm}\text{and}\hspace*{0.2cm}\hspace*{0.2cm}T\gg 1.
 \end{align*}
Hence
 \begin{align}\label{kp}
&\bigg(1-\frac{1}{p}\bigg)\int_{Q_T}\vert u\vert^p\tilde{\varphi}+C T^{-(\alpha_1-\gamma)}\int_{\mathbb{R}^N}u_1(x)\varphi_2(x)dx\nonumber\\
&\leqslant C\left( T^{-p'(2+\alpha_1-\gamma)+\frac{\theta N}{2}+1}+T^{-p'(\theta\delta+1+\alpha_2-\gamma)+\frac{\theta N}{2}+1}+T^{-p'(\theta\sigma+1-\gamma)+\frac{\theta N}{2}+1}\right)\nonumber\\
&\leqslant C\left( T^{-p'(2+\alpha_1-\gamma)+\frac{\theta N}{2}+1}+ 2T^{-p'(\theta\sigma+1-\gamma)+\frac{\theta N}{2}+1}\right)\nonumber\\
&\leqslant C T^{-p'(2+\alpha_1-\gamma)+\frac{\theta N}{2}+1}.
\end{align}
Under the condition\thinspace$u_1(x)\geqslant 0$, we obtain
\begin{align}
\frac{1}{p'}\int_{Q_T}\vert u\vert^p\tilde{\varphi}\leqslant C T^{-p'(2+\alpha_1-\gamma)+\frac{\theta N}{2}+1}.
\end{align}
Since
\begin{align*}
p\leqslant p*,
\end{align*}
we have to distinguish two cases.

In case $p<p*:$ if a solution of $(\ref{a})-(\ref{aa})$ exists globally, then taking $T\rightarrow+\infty$, we get
\[\lim_{T\rightarrow\infty}\int_{Q_T}\vert u\vert^p\tilde{\varphi}<0.\]
 Contradiction the fact that $\int_{0}^\infty\int_{\mathbb{R}^N}\vert u\vert^p\tilde{\varphi}\geqslant 0$.

In case\,$p=p*:$ we repeat the same calculation as above by  taking\thinspace$\varphi_2(x)=\psi(\frac{\vert x\vert}{B^{-1}T^{\theta/2}})$,  where \thinspace$1\ll B<T$\thinspace and when \thinspace$T$ goes to infinity we don't have $B$ goes to infinity at the same time, employing the H$\ddot{o}$lder's inequality instead of Young's, we obtain
\begin{align}\label{hh}
\int_{\Sigma_B}\vert u\vert\varphi_2 D^{2+\alpha_1-\gamma}_{t\vert T}\varphi_1\leqslant  B^{-\frac{N}{ p'}} \left(\int_{\Sigma_B}\vert u\vert^p\tilde{\varphi} \right)^{1/p}.
\end{align}
Using Lemma\thinspace\ref{kk} and  the H$\ddot{o}$lder inequality, we have
\begin{align}\label{rt}
\int_{\Sigma_B}\vert u(-\Delta)^{\delta}\varphi_2\vert D^{1+\alpha_2-\gamma}_{t\vert T}\varphi_1\leqslant C T^{-(\theta(\delta-\sigma)+\alpha_2)}B^{2\delta-\frac{N}{ p'}}\left(\int_{\Sigma_B} \vert u\vert^p \tilde{\varphi}\right)^{1/p},
\end{align}
and
\begin{align}\label{rp}
\int_{\Sigma_B}\vert u(-\Delta)^\sigma\varphi_2\vert D^{1-\gamma}_{t\vert T}\varphi_1\leqslant B^{2\sigma-\frac{N}{ p'}} \left(\int_{\Sigma_B} \vert u\vert^p \tilde{\varphi}\right)^{1/p},
\end{align}
where $\Sigma_B= [0,T]\times\text{supp}_x\varphi_2$. Combining\,$(\ref{hh})$, $(\ref{rt})$ with $(\ref{rp})$, we get
\begin{align*}
\int_{Q_T}\vert u\vert^p\tilde{\varphi}+C T^{-(\alpha_1-\gamma)}\int_{\mathbb{R}^N}u_1(x)\varphi_2(x)dx&\leqslant C  B^{-\frac{N}{\beta p'}}\left(\int_{\Sigma_B}\vert u\vert^p\tilde{\varphi} \right)^{1/p}+ C T^{-(\theta(\delta-\sigma)+\alpha_2)} B^{2\delta-\frac{N}{ p'}}\left(\int_{\Sigma_B} \vert u\vert^p \tilde{\varphi}\right)^{1/p}\\&+C B^{2\sigma-\frac{N}{ p'}}\left(\int_{\Sigma_B} \vert u\vert^p \tilde{\varphi}\right)^{1/p}.
\end{align*}
Thus, passing the limit as\thinspace$T\longrightarrow +\infty$\thinspace and then where\thinspace$B\longrightarrow +\infty$, with $p>\frac{N}{N-2\sigma}$ we have
\begin{align*}
\lim_{T\rightarrow\infty}\int_{Q_T}\vert u\vert^p\tilde{\varphi}< 0.
\end{align*}
 This leads to a contradiction.

\end{proof}
\section{Case of a $2\times 2$  system}
\hspace*{0.3cm} This part is concerned with the study of the following system 
\begin{align}\label{kio}
\begin{cases}
{\mathbf{D}}_{0|t}^{1+\alpha _{1}}u+(-\Delta)^{\sigma_1} u+(-\Delta )^{\delta _{1}}{\mathbf{D}}%
_{0|t}^{\alpha _{2}}u=I_{0|t}^{1-\gamma _{1}}|v|^{p},\hspace*{0.5cm}%
(t,x)\in (0,\infty )\times \mathbb{R}^{N}, \\
\\
{\mathbf{D}}_{0|t}^{1+\beta _{1}}v+(-\Delta)^{\sigma_2} v+(-\Delta )^{\delta _{2}}{\mathbf{D}}%
_{0|t}^{\beta _{2}}v=I_{0|t}^{1-\gamma _{2}}|u|^{q},\hspace*{0.5cm}%
(t,x)\in (0,\infty )\times \mathbb{R}^{N},  
\end{cases}
\end{align}%\begin{equation}
supplemented with the initial conditions
\begin{align}\label{aui}
\begin{cases}
u(0,x)=u_{0}(x), \;u_{t}(0,x)=u_{1}(x),\hspace*{0.5cm}x\in \mathbb{R}^{N},\\
v(0,x)=v_{0}(x), \;v_{t}(0,x)=v_{1}(x),\hspace*{0.5cm}x\in \mathbb{R}^{N},
\end{cases}
\end{align}
where\thinspace $p>1$, $q>1$, $0<\gamma_1<\alpha _{2}<\alpha _{1}<1$, $0<\gamma_2<\beta _{2}<\beta _{1}<1$ and  $0<\sigma_i <\delta_i <1$.

\begin{theo}
Assume that $u_0,\;v_0=0$, whereas $u_1,\;v_1\in L^1(\mathbb{R}^N)$ and $u_1,\;v_1\geqslant 0$. If  
\begin{equation*}
N< \max \left\{ \frac{(2+\beta_1-\gamma_2)+\frac{1}{p}+q(1+\alpha_1-\gamma_1)}{\frac{(\beta_1+1)}{2p'\sigma_2}+\frac{(\alpha_1+1) q}{2q'\sigma_1}};\frac{(2+\alpha_1-\gamma_1)+\frac{1}{q}+p(1+\beta_1-\gamma_2)}{\frac{(\alpha_1+1)}{2q'\sigma_1}+\frac{(\beta_1+1) p}{2p'\sigma_2}}
\right\},
 \end{equation*}%
then any  solution $(u,v)$ to\thinspace $(\ref{kio})-(\ref{aui})$ blows-up in a finite time.
\end{theo}

\begin{proof}
 The proof proceeds by contradiction. Let
 \begin{align*}
\phi(t,x)=D^{1-\gamma_i}_{t\vert T}\tilde{\varphi}(t,x)=D^{1-\gamma_i}_{t\vert T}\left(\varphi_1(t)\varphi_2(x) \right),\hspace*{0.3cm} i=1,2,
\end{align*}
where $\varphi_1$ is defined as in $(\ref{ob})$ however with condition $\eta>\bigg\lbrace \frac{p}{p-1}(2+\alpha_1-\gamma_1), \frac{p}{p-1}(2+\beta_1-\gamma_2)\bigg\rbrace $
and\, $\varphi_2(x)=\psi\left(\frac{\vert x\vert}{T^{\theta_i/2}} \right)$\,is defined above.\\

The weak solutions to system $(\ref{kio})-(\ref{aui})$\,reads as
\begin{align}\label{in}
&\int_{Q_T}\vert v\vert^p\tilde{\varphi}+\int_{Q_T}u_1\varphi_2D^{1+\alpha_1-\gamma_1}_{t\vert T}\varphi_1\nonumber\\&=\int_{Q_T}u\varphi_2D^{2+\alpha_1-\gamma_1}_{t\vert T}\varphi_1+\int_{Q_T}u(-\Delta)^{\delta_1}\varphi_2 D^{1+\alpha_2-\gamma_1}_{t\vert T}\varphi_1+\int_{Q_T}u(-\Delta)^{\sigma_1}\varphi_2D^{1-\gamma_1}_{t\vert T}\varphi_1,
\end{align}
and
\begin{align}
&\int_{Q_T}\vert u\vert^q\tilde{\varphi}+\int_{Q_T}v_1\varphi_2D^{1+\beta_1-\gamma_2}_{t\vert T}\varphi_1\nonumber\\&=\int_{Q_T}v\varphi_2D^{2+\beta_1-\gamma_2}_{t\vert T}\varphi_1+\int_{Q_T}v(-\Delta)^{\delta_2}\varphi_2 D^{1+\beta_2-\gamma_2}_{t\vert T}\varphi_1+\int_{Q_T}v(-\Delta)^{\sigma_2}\varphi_2D^{1-\gamma_2}_{t\vert T}\varphi_1.
\end{align} 
Using the H$\ddot{o}$lder inequality, we obtain
\begin{align}\label{inn}
\int_{Q_T}u\varphi_2D^{2+\alpha_1-\gamma_1}_{t\vert T}\varphi_1\leqslant \left(\int_{Q_T}\vert u\vert^q\tilde{\varphi}\right)^{1/q}\left(\int_{Q_T}\varphi_2\varphi_1^{-\frac{q'}{q}}\vert D^{2+\alpha_1-\gamma_1}_{t\vert T}\varphi_1\vert^{q'} \right)^{1/q'},
\end{align}
\begin{align}\label{innn}
&\int_{Q_T}u(-\Delta)^{\delta_1}\varphi_2 D^{1+\alpha_2-\gamma_1}_{t\vert T}\varphi_1\nonumber\\
&\leqslant\left(\int_{Q_T}\vert u\vert^q\tilde{\varphi}\right)^{1/q}\left(\int_{Q_T}\varphi_2^{-\frac{q'}{q}}\vert(-\Delta)^{\delta_1}\varphi_2\vert^{q'}\varphi_1^{-\frac{q'}{q}}\vert D^{1+\alpha_2-\gamma_1}_{t\vert T}\varphi_1\vert^{q'}\right)^{1/q'},
\end{align}
and
\begin{align}\label{innnn}
\int_{Q_T}u(-\Delta)^{\sigma_1}\varphi_2 D^{1-\gamma_1}_{t\vert T}\varphi_1
\leqslant\bigg(\int_{Q_T}\vert u\vert^q\tilde{\varphi}\bigg{)}^{1/q}\left(\int_{Q_T}\varphi_2^{-\frac{q'}{q}}\vert(-\Delta)^{\sigma_1}\varphi_2\vert^{q'}\varphi_1^{-\frac{q'}{q}}\vert D^{1-\gamma_1}_{t\vert T}\varphi_1\vert^{q'}\right)^{1/q'}.
\end{align}
Taking into account the above relation (\ref{in}), (\ref{inn}), (\ref{innn}) and\thinspace$(\ref{innnn})$, we find
\begin{align}\label{obb}
\int_{Q_T}\vert v\vert^p\tilde{\varphi}+C T^{-(\alpha_1-\gamma_1)}\int_{\mathbb{R}^N}u_1\varphi_2\leqslant \left(\int_{Q_T}\vert u\vert^q\tilde{\varphi}\right)^{1/q} \mathcal{A},
\end{align}
we have set
\begin{align*}
\mathcal{A}=\left(\int_{Q_T}\varphi_2\varphi_1^{-\frac{q'}{q}}\vert D^{2+\alpha_1-\gamma_1}_{t\vert T}\varphi_1\vert^{q'} \right)^{1/q'}&+\left(\int_{Q_T}\varphi_2^{-\frac{q'}{q}}\vert(-\Delta)^{\delta_1}\varphi_2\vert^{q'}\varphi_1^{-\frac{q'}{q}}\vert D^{1+\alpha_2-\gamma_1}_{t\vert T}\varphi_1\vert^{q'}\right)^{1/q'}\\&+\left(\int_{Q_T}\varphi_2^{-\frac{q'}{q}}\vert(-\Delta)^{\sigma_1}\varphi_2\vert^{q'}\varphi_1^{-\frac{q'}{q}}\vert D^{1-\gamma_1}_{t\vert T}\varphi_1\vert^{q'}\right)^{1/q'}.
\end{align*}
Similarly, we get 
\begin{align}\label{obk}
\int_{Q_T}\vert u\vert^q\tilde{\varphi}+C T^{-(\beta_1-\gamma_2)}\int_{\mathbb{R}^N}v_1\varphi_2\leqslant \left(\int_{Q_T}\vert v\vert^p\tilde{\varphi}\right)^{1/p} \mathcal{B},
\end{align}
with
\begin{align*}
\mathcal{B}=\left(\int_{Q_T}\varphi_2\varphi_1^{-\frac{p'}{p}}\vert D^{2+\beta_1-\gamma_2}_{t\vert T}\varphi_1\vert^{p'} \right)^{1/p'}&+\left(\int_{Q_T}\varphi_2^{-\frac{p'}{p}}\vert(-\Delta)^{\delta_2}\varphi_2\vert^{p'}\varphi_1^{-\frac{p'}{p}}\vert D^{1+\beta_2-\gamma_2}_{t\vert T}\varphi_1\vert^{p'}\right)^{1/p'}\\&+\left(\int_{Q_T}\varphi_2^{-\frac{p'}{p}}\vert(-\Delta)^{\sigma_2}\varphi_2\vert^{p'}\varphi_1^{-\frac{p'}{p}}\vert D^{1-\gamma_2}_{t\vert T}\varphi_1\vert^{p'}\right)^{1/p'}.
\end{align*}
Therefore, as $u_1, v_1\geqslant 0$, we obtain
\begin{align}\label{obbb}
\int_{Q_T}\vert v\vert^p\tilde{\varphi}\leqslant \left(\int_{Q_T}\vert u\vert^q\tilde{\varphi}\right)^{1/q} \mathcal{A},
\end{align}
and
\begin{align}\label{obkk}
\int_{Q_T}\vert u\vert^q\tilde{\varphi}\leqslant \left(\int_{Q_T}\vert v\vert^p\tilde{\varphi}\right)^{1/p} \mathcal{B}.
\end{align}
 Now, combining\,$(\ref{obbb})$\,and\,$(\ref{obkk})$, we write
\begin{align}
\begin{cases}
\bigg(\int_{Q_T}\vert v\vert^p\tilde{\varphi}\bigg)^{1-\frac{1}{pq}}\leqslant\mathcal{B}^{\frac{1}{q}}\mathcal{A},\\
\\
\bigg(\int_{Q_T}\vert u\vert^q\tilde{\varphi}\bigg)^{1-\frac{1}{pq}}\leqslant\mathcal{A}^{\frac{1}{p}}\mathcal{B}.
\end{cases}
\end{align}
Using Lemma\thinspace\ref{kk}, Lemma\thinspace\ref{ui} and making  the change of variables
\begin{align*}
x=\xi T^{\frac{\theta_1}{2}} \hspace*{0.2cm}\text{with}\hspace*{0.2cm} \theta_1=\frac{\alpha_1+1}{\sigma_1}\text{ in}\hspace*{0.2cm} \mathcal{A},
\end{align*}
\begin{align*}
x=\xi T^{\frac{\theta_2}{2}} \hspace*{0.2cm}\text{with}\hspace*{0.2cm} \theta_2=\frac{\beta_1+1}{\sigma_2}\text{ in}\hspace*{0.2cm} \mathcal{B}.
\end{align*}
 We obtain the estimates
\begin{align}\label{poi}
\begin{cases}
\bigg(\int_{Q_T}\vert v\vert^p\tilde{\varphi}\bigg)^{1-\frac{1}{pq}}\leqslant T^{l_1},\\
\\
\bigg(\int_{Q_T}\vert u\vert^q\tilde{\varphi}\bigg)^{1-\frac{1}{pq}}\leqslant T^{l_2},
\end{cases}
\end{align}
where
\begin{align*}
l_1=\bigg({-(2+\beta_1-\gamma_2)+\frac{1}{p'}(\frac{\theta_2 N}{2}+1)}\bigg)\frac{1}{q}-(2+\alpha_1-\gamma_1)+\frac{1}{q'}(\frac{\theta_1 N}{2}+1),
\end{align*}
and
\begin{align*}
l_2=\bigg({-(2+\alpha_1-\gamma_1)+\frac{1}{q'}(\frac{\theta_1 N}{2}+1)}\bigg)\frac{1}{p}-(2+\beta_1-\gamma_2)+\frac{1}{p'}(\frac{\theta_2 N}{2}+1).
\end{align*}
Hence, by taking the limit as $T\rightarrow\infty$ in $(\ref{poi})$, we obtain
\begin{align*}
\begin{cases}
\int_{0}^\infty\int_{\mathbb{R}^N}\vert v\vert^p\tilde{\varphi}<0,\\
\\
\int_{0}^\infty\int_{\mathbb{R}^N}\vert u\vert^q\tilde{\varphi}<0,
\end{cases}
\end{align*}
which is a contradiction. Then $(u,v)$ cannot be a global solution.
\end{proof}
\section{Necessary conditions for local and global existence}

\setcounter{section}{5} \hspace*{0.3cm}In this part, we establish the
necessary conditions for the local and global existence of solutions to
 the problem $(\ref{a})- (\ref{aa})$.

\begin{theo}
Let $u_0=0$ and let $u$\thinspace\ be a local solution to problem $(\ref{a})- (\ref{aa})$, then
exists positive constant $C$, such that 
\begin{equation*}
\inf_{|x|\longrightarrow \infty
}u_{1}(x)\leqslant C T^{-\frac{p}{p-1}(2+\alpha _{1}-\gamma )+(1+\alpha _{1}-\gamma )}.
\end{equation*}
\end{theo}

\begin{proof}
Using the Young inequality in the right hand side of \;\((\ref{p})\), we obtain
\begin{align*}
&\int_{Q_T}u_1\varphi_2 D^{1+\alpha_1-\gamma}_{t\vert T}\varphi_1\\
&\leqslant\frac{3^{p'-1}}{p'}\int_{Q_T}\varphi_1^{-p'/p}\varphi_2^{-p'/p}\left( \varphi_2^{p'}\vert D^{2+\alpha_1-\gamma}_{t\vert T}\varphi_1\vert^{p'}+\vert(-\Delta)^{\delta}\varphi_2D^{1+\alpha_2-\gamma}_{t\vert T}\varphi_1\vert^{p'}+\vert(-\Delta)^\sigma\varphi_2D^{1-\gamma}_{t\vert T}\varphi_1\vert^{p'}\right).
\end{align*}
 Let $\psi$ the first eigenfunction of $(-\Delta)^\theta$ with $\lambda_\theta $  the first eigenvalue on $\Omega$\,\cite{vp}
\begin{equation}
\begin{cases}
(-\Delta)^\theta\psi(x)=\lambda_\theta \psi(x), \hspace*{0.5cm}&x\in \Omega,\\
\psi(x)=0, \hspace*{0.3cm}& x\in \mathbb{R}^N\setminus \Omega,\\
\psi(x)\geqslant0, \hspace*{0.3cm}& x\in \mathbb{R}^N,\\
\int_\Omega\psi(x)dx=1,
\end{cases}
\end{equation}
where $\Omega:=\lbrace x\in\mathbb{R}^N\hspace*{0.1cm}:\hspace*{0.1cm}1< \vert x\vert <2\rbrace $ and $\theta=\lbrace \sigma, \delta\rbrace$.\\
 Choose $\varphi_2(x)=\psi(x/R)$ and $\xi=x/R$, so that $(-\Delta)^\theta\varphi_2(x)=(-\Delta)^\theta \psi(x/R)=R^{-2\theta}(-\Delta)^\theta\psi(\xi)$, we obtain
\begin{align*}
T^{1-(1+\alpha_1-\gamma)}\int_{\Omega}u_1(R\xi)\psi(\xi) 
\leqslant C\left(T^{1-p'(2+\alpha_2-\gamma)}+\lambda_\delta R^{-2p'\delta}T^{1-p'(1+\alpha_2-\gamma)}+\lambda_\sigma R^{-2p'\sigma}T^{1-p'(1-\gamma)} \right)\int_{\Omega} \psi(\xi) .
\end{align*}
Thus
\begin{align}\label{kir}
&T^{-(1+\alpha_1-\gamma)}\inf_{\vert \xi\vert>1 }u_1(R\xi) \int_{\Omega}\psi(\xi)\nonumber\\
&\leqslant C\left( T^{-p'(2+\alpha_1-\gamma)}+R^{-2p'\delta}T^{-p'(1+\alpha_2-\gamma)}+R^{-2\sigma p'}T^{-p'(1-\gamma)}\right) \int_{\Omega} \psi(\xi).
\end{align}
By passing to the limit in\;$(\ref{kir})$, as\,$R\rightarrow \infty$, we get 
\begin{align}\label{kira}
T^{-(1+\alpha_1-\gamma)}\inf_{\vert x\vert\longrightarrow\infty }u_1(x)\leqslant CT^{-p'(2+\alpha_1-\gamma)}.
\end{align}
\end{proof}

\begin{corol}
Assume that \,$\lim_{\vert x\vert\longrightarrow\infty}\inf u_1(x)=+\infty$, then
problem $(\ref{a})- (\ref{aa})$ does not admit any local solution.
\end{corol}

\begin{corol}
Suppose that problem\,$(\ref{a})- (\ref{aa})$ has a global solution, then $\lim_{\vert
x\vert\longrightarrow\infty}\inf u_1(x)=0$.
\end{corol}

\begin{proof}
 The proof proceeds by contradiction. Suppose 
 
 \begin{align*}
C= \lim_{\vert x\vert\longrightarrow\infty}\inf u_1>0.
 \end{align*}
Using $(\ref{kira})$ for $T>1$, we get  
 \begin{align*}
T^{-(1+\alpha_1-\gamma)+p'(2+\alpha_1-\gamma)}\leqslant C.
 \end{align*}
We get a contradiction when $T\rightarrow\infty$ .
 \end{proof}

\begin{theo}
Suppose\thinspace $(\ref{a})- (\ref{aa})$ admit a global weak solution, then  exists
 positive constant\thinspace $C$\thinspace\ such that\newline

\begin{equation*}
\lim_{|x|\longrightarrow \infty }\left[ \inf u_{1}\vert x\vert^{\frac{2\delta}{1+\alpha_1}[p'(2+\alpha_1-\gamma)-(1+\alpha_1-\gamma)]}
\right] \leqslant C.
\end{equation*}
\end{theo}

\begin{proof}
Using the relation\,$(\ref{kir})$, we have

 \begin{align*}
 \inf_{\vert \xi\vert>1 }u_1(R\xi) \int_{\Omega}\psi(\xi)\leqslant C\bigg{(} T^{(1+\alpha_1-\gamma)-p'(2+\alpha_1-\gamma)}&+R^{-2\delta p'}T^{(1+\alpha_1-\gamma)-p'(1+\alpha_2-\gamma)}\\
 &+R^{-2\sigma p'}T^{(1+\alpha_1-\gamma)-p'(1-\gamma)}\bigg{)} \int_{\Omega}\psi(\xi),
 \end{align*}
 which implies that
 \begin{align*}
 \inf_{\vert x\vert>R }u_1(x) \int_{\Omega}\psi(\xi)&\leqslant CT^{(1+\alpha_1-\gamma)-p'(2+\alpha_1-\gamma)}+ T^{(1+\alpha_1-\gamma)-p'(1-\gamma)}R^{-2\delta p'}\int_{\Omega}\psi(\xi). 
 \end{align*}
 The right-hand side have a minimum at
 \[T^{\natural}=\left[\frac{p'(2+\alpha_1-\gamma)-(1+\alpha_1-\gamma)}{(1+\alpha_1-\gamma)-p'(1-\gamma)}\right]^{\frac{1}{p'(1+\alpha_1)}} R^{\frac{2\delta}{1+\alpha_1}},\]
whereupon
\begin{align*}
 &\inf_{\vert x\vert>R }\left[u_1\vert x\vert^{\frac{2\delta}{1+\alpha_1}[p'(2+\alpha_1-\gamma)-(1+\alpha_1-\gamma)]}\right] \int_{\Omega}\vert R\xi\vert^{-\frac{2\delta}{1+\alpha_1}[p'(2+\alpha_1-\gamma)-(1+\alpha_1-\gamma)]}\psi(\xi)\\&\leqslant C2^{\frac{2\delta}{1+\alpha_1}[p'(2+\alpha_1-\gamma)-(1+\alpha_1-\gamma)]}\int_{\Omega}\vert R\xi\vert^{-\frac{2\delta}{1+\alpha_1}[p'(2+\alpha_1-\gamma)-(1+\alpha_1-\gamma)]}\psi(\xi). 
 \end{align*}
 Finally, dividing by
 \begin{align*}
 \int_{\Omega}\vert  R\xi\vert^{-\frac{2\delta}{1+\alpha_1}[p'(2+\alpha_1-\gamma)-(1+\alpha_1-\gamma)]}\psi(\xi)>0,
 \end{align*}
 we obtain
 \begin{equation*}
\lim_{|x|\longrightarrow \infty }\left[ \inf u_{1}\vert x\vert^{\frac{2\delta}{1+\alpha_1}[p'(2+\alpha_1-\gamma)-(1+\alpha_1-\gamma)]}
\right] \leqslant C.
\end{equation*}
 \end{proof}
%\bibliographystyle{siam}
%\bibliography{bib}
%%%%%%%%%% put authors' addresses here, in \it %%%%%%%%
 %\bigskip \smallskip
 \it
 
\end{document}